\documentclass[11pt,twoside]{article}
\usepackage{mathrsfs}
\usepackage{amssymb}
\usepackage{amsmath}
\usepackage[all]{xy}
\usepackage{amsthm}
 \usepackage{color}
\numberwithin{equation}{section}

\newtheorem{theorem}{Theorem}[section]

\newtheorem{definition}[theorem]{Definition}
\newtheorem{remark}[theorem]{Remark}
\newtheorem{lemma}[theorem]{Lemma}
\newtheorem{example}[theorem]{Example}

\newtheorem{corollary}[theorem]{Corollary}

\newcommand{\edge}{\ar@{-}}

\newcommand{\pf}{\noindent\begin {proof}}
\newcommand{\epf}{\end{proof}}

\newcommand{\extdim}{\mbox{\rm ext.dim}}

\def\Im{\mathop{\rm Im}\nolimits}
\def\Ker{\mathop{\rm Ker}\nolimits}

\def\La{\mathop{\rm \Lambda}\nolimits}

\def\mod{\mathop{\rm mod}\nolimits}

\def\pd{\mathop{\rm pd}\nolimits}

\def\max{\mathop{\rm max}\nolimits}

\def\sup{\mathop{\rm sup}\nolimits}
\def\inf{\mathop{\rm inf}\nolimits}
\def\add{\mathop{\rm add}\nolimits}

\def\gldim{\mathop{\rm gl.dim}\nolimits}

\def\rad{\mathop{{\rm rad}}\nolimits}
\def\top{\mathop{{\rm top}}\nolimits}

\def\dim{\mathop{\rm dim}\nolimits}
\def\tridim{\mathop{\rm tri.dim}\nolimits}

\def\sup{\mathop{\rm sup}\nolimits}
\def\lim{\mathop{\underrightarrow{\rm lim}}\nolimits}

\def\mod{\mathop{\rm mod}\nolimits}

\def\pd{\mathop{\rm pd}\nolimits}
\def\max{\mathop{\rm max}\nolimits}

\def\sup{\mathop{\rm sup}\nolimits}
\def\inf{\mathop{\rm inf}\nolimits}
\def\add{\mathop{\rm add}\nolimits}

\def\gldim{\mathop{\rm gl.dim}\nolimits}

\def\rad{\mathop{{\rm rad}}\nolimits}

\def\top{\mathop{{\rm top}}\nolimits}

\def\dim{\mathop{\rm dim}\nolimits}

\def\sup{\mathop{\rm sup}\nolimits}
\def\lim{\mathop{\underrightarrow{\rm lim}}\nolimits}

\def\repdim{\mathop{\rm rep.dim}\nolimits}

\def\V{\mathop{\rm \mathcal{V}}\nolimits}
\def\T{\mathop{\rm \mathcal{T}}\nolimits}
\def\I{\mathop{\rm \mathcal{I}}\nolimits}

\def\LL{\mathop{\rm LL}\nolimits}

\title{ \bf The derived dimensions of $(m,n)$-Igusa-Todorov algebras
\thanks{2020 Mathematics Subject Classification: 18G20, 16E10, 16E35 .}
\thanks{Keywords:
Igusa-Todorov, the bounded derived category, radical layer length, extension dimension, triangulated categories. }}
\vspace{0.2cm}

\author { \ Junling  Zheng\thanks{Email: zhengjunling@cjlu.edu.cn}
\\
{\it \scriptsize  Department of Mathematics, China Jiliang University, Hangzhou, 310018, P. R. China
}}

\date{ }

\begin{document}

\baselineskip=16pt


\maketitle

\begin{abstract}

We give an upper bound for the dimension of the bounded derived categories of
$(m,n)$-Igusa-Todorov algebras which is a generalization of $n$-Igusa-Todorov algebras,
where $m,n$ are two nonnegative integers. As an applications, we get a new upper bound for the
dimension of bounded derived categories in terms of the projective dimensions of certain of simple modules as well as
radical layer length of artin algebra $\Lambda$.
\end{abstract}

\pagestyle{myheadings}
\markboth{\rightline {\scriptsize  J. L. Zheng\emph{}}}
         {\leftline{\scriptsize
         $(m,n)$-Igusa-Todorov algebras}}


\section{Introduction} 

Given a triangulated category $\mathcal{T}$, Rouquier introduced in
 \cite{rouquier2006representation} the dimension $\tridim\mathcal{T}$ of $\mathcal{T}$
under the idea of Bondal and van den Bergh in \cite{bondal2003generators}.
 This dimension and the infimum of the Orlov spectrum of $\mathcal{T}$
coincide, see \cite{orlov2009remarks,ballard2012orlov}. Roughly speaking,
 it is an invariant that measures how quickly
 the category can be built from one object.
Many authors have studied the upper bound of $\tridim \mathcal{T}$,
 see \cite{ballard2012orlov, bergh2015gorenstein, chen2008algebras,
  han2009derived, krause2006rouquier, oppermann2012generating,
   rouquier2006representation, rouquier2008dimensions,
   zheng2020thedimension,zheng2020extension,zheng2020upper} and so on.
There are a lot of triangulated categories having infinite dimension,
for instance, Oppermann and \v St'ov\'\i \v cek proved
in \cite{oppermann2012generating} that
all proper thick subcategories of the bounded derived category
 of finitely generated modules over a Noetherian algebra
containing perfect complexes have infinite dimension.

Let $\Lambda$ be an artin algebra. Let $\mod \Lambda$ be the
category of finitely generated right $\Lambda$-modules
and let $D^{b}(\mod \Lambda)$ be the bounded derived category of
$ \Lambda$. For convenience, $\tridim D^{b}(\mod \Lambda)$ also is said to
be the derived dimension of artin algebra $\Lambda$(for example, see \cite{chen2008algebras}).

The dimension of triangulated category plays
an important role in representation theory(\cite{ballard2012orlov,
bergh2015gorenstein,
chen2008algebras,
han2009derived,
oppermann2012generating,
rouquier2006representation,
rouquier2008dimensions,
zheng2020upper}).
For example, it can be used to study the representation dimension of
artin algebras (\cite{rouquier2006representation,oppermann2009lower}).
Similar to the dimension of triangulated categories, the (extension) dimension of an abelian category
was introduced by Beligiannis in \cite{beligiannis2008some}, also see \cite{dao2014radius}. The size of the extension dimension reflects how far an artin algebra is from a finite representation type, some relate result can
be see \cite{zheng2020extension,zheng2020thedimension,zheng2021matrix} and so on.
We denote $\extdim \mod \Lambda$ by the extension dimension of $\mod \Lambda$,
 and $\LL(\Lambda)$ by the Loewy length of $\Lambda$,
  $\gldim \Lambda$ by the global dimension of $\Lambda$,
  $\repdim \Lambda$ by the representation dimension of $\Lambda$ (see \cite{auslander1999representation}).
 The upper bounds for the dimensions of
 the bounded derived category of
$\mod \Lambda$ can be given in terms of
 $\LL(\Lambda)$,  $\gldim\Lambda$, $\repdim \Lambda$, and $\extdim \mod \Lambda$.
Given an non-semisimple artin algebra $\Lambda$, these dimensions have the following relation
$$\extdim \mod \Lambda \leqslant \max\{\LL(\Lambda)-1, \gldim \Lambda, \repdim \Lambda-2\}.$$

In \cite{wei2009finitistic}, Wei introduced the notion of $n$-Igusa-Todorov algebra. The relation between extension dimension and $0$-Igusa-Todorov algebra is that artin algebra $\Lambda$ is
$0$-Igusa-Todorov algebra if and only if $\extdim \mod \Lambda \leqslant 1$(see \cite[Proposition 3.15]{zheng2020extension}).
On the other hand, the authors in \cite{zheng2020thedimension} give an upper bound for the derived dimension of $n$-Igusa-Todorov.
For the sake of convenience, we will introduce the notion of $(m,n)$-Igusa-Todorov algebras, where
$m,n$ are two nonnegative integers, which is a generalization of $n$-Igusa-Todorov algebra.
Moreover, we also give an upper bound for the dimension of bounded derived categories of
$(m,n)$-Igusa-Todorov algebras in terms of $m$ and $n$.

For a length-category $\mathcal{C}$,
generalizing the Loewy length, Huard, Lanzilotta and Hern\'andez introduced
in \cite{huard2013layer,huard2009finitistic} the (radical) layer
length associated with a torsion pair, which is a new measure for objects of
$\mathcal{C}$. Let $\Lambda$ be an artin algebra and $\mathcal{V}$
a set of some simple modules in $\mod \Lambda$.
Let $t_{\mathcal{V}}$ be the torsion radical of a
torsion pair associated with $\mathcal{V}$ (see Section 4 for details). We use
$\ell\ell^{t_{\mathcal{V}}}(\Lambda)$ to
denote the $t_{\mathcal{V}}$-radical layer length of $\Lambda$.
For a module $M$ in $\mod \Lambda$, we use $\pd M$ to denote the
 projective dimensions of $M$;
in particular, set $\pd M=-1$ if $M=0$.
 For a subclass $\mathcal{B}$ of $\mod \Lambda$,
  the {\bf projective dimension} $\pd\mathcal{B}$
of $\mathcal{B}$ is defined as
\begin{equation*}
\pd \mathcal{B}=
\begin{cases}
\sup\{\pd M\;|\; M\in \mathcal{B}\}, & \text{if} \;\; \mathcal{B}\neq \varnothing;\\
-1,&\text{if} \;\; \mathcal{B}=\varnothing.
\end{cases}
\end{equation*}
 Note that $\mathcal{V}$ is a finite set. So, if each simple module
in $\mathcal{V}$ has finite projective dimension,
 then $\pd \mathcal{V}$ attains its (finite) maximum.

Now, let us list some results about the upper bound of
the dimension of bounded derived categries.
\begin{theorem} \label{thm1.1}
Let $\Lambda$ be an artin algebra and $\mathcal{V}$
a set of some simple modules in $\mod \Lambda$. We have
\begin{enumerate}
\item[$(1)$] {\rm (\cite[Proposition 7.37]{rouquier2008dimensions})}
 $\tridim D^{b}(\mod \Lambda) \leqslant \LL(\Lambda)-1;$
\item[$(2)$] {\rm (\cite[Proposition 7.4]{rouquier2008dimensions}
and \cite[Proposition 2.6]{krause2006rouquier})}
 $\tridim D^{b}(\mod \Lambda) \leqslant \gldim \Lambda;$
\item[$(3)$] {\rm (\cite[Theorem 3.8]{zheng2020upper})}
  $\tridim D^{b}(\mod \Lambda) \leqslant
  (\pd\mathcal{V}+2)(\ell\ell^{t_{\mathcal{V}}}(\Lambda)+1)-2;$
\item[$(4)$] {\rm (\cite{zheng2020thedimension})}
$\tridim D^{b}(\mod \Lambda) \leqslant
2(\pd\mathcal{V}+\ell\ell^{t_{\mathcal{V}}}(\Lambda))+1;$

\item[$(5)$] {\rm (\cite{zheng2021radicalfinite})} if $\ell\ell^{t_{\V}}(\Lambda_{\Lambda})\leqslant 2$,
$\tridim D^{b}(\mod \Lambda) \leqslant
\pd\mathcal{V}+3.$
\end{enumerate}
\end{theorem}

The aim of this paper is to prove the following
\begin{theorem}
  Let $\Lambda$ be an $(m,n)$-Igusa-Todorov algebra.
  Then $\tridim D^{b}(\mod \Lambda) \leqslant 2m+n+1.$
  \end{theorem}

\begin{theorem} \label{thm1.2}
Let $\Lambda$ be an artin algebra and $\mathcal{V}$
a set of some simple modules in $\mod \Lambda$.
 If $\ell\ell^{t_{\V}}(\Lambda_{\Lambda})\geqslant 2$,
then $\tridim D^{b}(\mod \Lambda) \leqslant
2\ell\ell^{t_{\V}}(\Lambda_{\Lambda})+\pd\mathcal{V}-1.$
\end{theorem}
By Theorem \ref{thm1.1}(5) and Theorem \ref{thm1.2}, we have
\begin{corollary}
Let $\Lambda$ be an artin algebra and $\mathcal{V}$
a set of some simple modules in $\mod \Lambda$.
Then $\tridim D^{b}(\mod \Lambda) \leqslant \max\{
2\ell\ell^{t_{\V}}(\Lambda_{\Lambda})+\pd\mathcal{V}-1,\pd\mathcal{V}+3\}.$
\end{corollary}
    We also give examples to explain
     our results. Sometimes,
     we may be able to get a better upper
     bound for the dimension of
   the bounded derived category of $\Lambda$.
\begin{corollary}\label{cor1}
If artin algebra $\Lambda$ is $n$-Igusa-Todorov algebra.
 Then
$\tridim D^{b}(\mod \Lambda) \leqslant n+3.$
\end{corollary}

  \section{Preliminaries}

  \subsection{$(m,n)$-Igusa-Todorov algebras}
  In order to illustrate our main results, we introduce the following definition.
  \begin{definition}\label{mnITalgebra}
    {\rm
    For two nonnegative integers $m, n$. The artin algebra $\Lambda$ is said
    to be $(m,n)$-Igusa-Todorov algebra if there is a module
     $V\in \mod \Lambda$
    such that for any module $M$ there exists an exact sequence
    $$0
    \longrightarrow V_{m}
    \longrightarrow V_{m-1}
    \longrightarrow
    \cdots
    \longrightarrow V_{1}
    \longrightarrow V_{0}
    \longrightarrow \Omega^{n}(M)
    \longrightarrow 0
    $$
  where $V_{i} \in [V]_{1}$ for each $0 \leqslant  i \leqslant m $.
  Such a module $V$ is said to be $(m,n)$-Igusa-Todorov module.
  }
  \end{definition}
  \begin{remark}\label{mnITremark}
    {\rm
    $(1)$ By Definition \ref{mnITalgebra}, we know that $(1,n)$-Igusa-Todorov algebras is the same as $n$-Igusa-Todorov
    algebras (see \cite{wei2009finitistic}).

    $(2)$ $(0,n)$-Igusa-Todorov algebras is the same as $n$-syzygy-finite
    algebras (see \cite{wei2009finitistic}).

    $(3)$ $(m,n)$-Igusa-Todorov algebras are $(m+i,n-i)$-Igusa-Todorov algebras,
    also are $(m,n+i)$-Igusa-Todorov algebras,
    where $(n-i)$ is non-negative.

    $(4)$ If $\gldim \Lambda <\infty$, then $\Lambda$ is
    $(\gldim \Lambda, 0)$-Igusa-Todorov algebra.

    $(5)$ Set $\LL(\Lambda)=n\geqslant 2$. By \cite[Lemma 3.1]{zheng_huang_2020finitistic},
    we konw that, for each $M\in \mod \Lambda$,
    we have the following exact sequence
    $$0\rightarrow M_{n-1}\rightarrow M_{n-2}\rightarrow
     \cdots \rightarrow M_{1}\rightarrow M_{0}\rightarrow M\rightarrow 0$$
    be an exact sequence in $\mod\Lambda$,
    where $M_{i}\in \add (\Lambda / \rad^{n-i}\Lambda)$.
    Then $\Lambda$ is
    $(\LL(\Lambda)-1, 0)$-Igusa-Todorov algebra by Definition \ref{mnITalgebra}.
Since $M_{0}$ is projective, we can get that $\Lambda$ is
    $(\LL(\Lambda)-2, 1)$-Igusa-Todorov algebra. Then $\Lambda$ is
    $(\LL(\Lambda)-2, 2)$-Igusa-Todorov algebra by Remark \ref{mnITremark}(3).


    }

  \end{remark}

  \subsection{The dimension of triangulated category}
  We recall some notions from
  \cite{rouquier2006representation,rouquier2008dimensions,
  oppermann2009lower}.
  Let $\T$ be a triangulated category and $\I \subseteq {\rm Ob}\T$.
  Let $\langle \I \rangle_{1}$ be the full subcategory
  consisting of $\T$
  of all direct summands of finite direct sums of shifts of
  objects in $\I$.
  Given two subclasses $\I_{1}, \I_{2}\subseteq {\rm Ob}\T$,
  we denote $\I_{1}*\I_{2}$
  by the full subcategory of all extensions between them, that is,
  $$\I_{1}*\I_{2}=\{ X\mid  X_{1} \longrightarrow X
  \longrightarrow X_{2}\longrightarrow X_{1}[1]\;
  {\rm with}\; X_{1}\in \I_{1}\; {\rm and}\; X_{2}\in \I_{2}\}.$$
  Write $\I_{1}\diamond\I_{2}:=\langle\I_{1}*\I_{2} \rangle_{1}.$
  Then $(\I_{1}\diamond\I_{2})\diamond\I_{3}=\I_{1}
  \diamond(\I_{2}\diamond\I_{3})$
  for any subclasses $\I_{1}, \I_{2}$ and $\I_{3}$
  of $\T$ by the octahedral axiom.
  Write
  \begin{align*}
  \langle \I \rangle_{0}:=0,\;
  \langle \I \rangle_{n+1}:=\langle \I
  \rangle_{n}\diamond\langle \I \rangle_{1}\;{\rm for\; any \;}
  n\geqslant 1.
  \end{align*}

  \begin{definition}{\rm
    (\cite[Definiton 3.2]{rouquier2006representation})\label{tri.dimenson2.1}
  The {\bf dimension} $\tridim \T$ of a triangulated category $\T$
  is the minimal $d$ such that there exists an object $M\in \T$ with
  $\T=\langle M \rangle_{d+1}$. If no such $M$ exists for any $d$,
  then we set $\tridim \T=\infty.$
  }
  \end{definition}

  \begin{lemma}{\rm (\cite[Lemma 7.3]{psaroudakis2014homological})}\label{lem2.5}
  Let $\T$ be a triangulated category and let $X, Y$ be
   two objects of $\T$.
  Then
  $$\langle X \rangle _{m}\diamond \langle Y \rangle _{n}
  \subseteq \langle X\oplus Y \rangle _{m+n}$$
  for any $m,n \geqslant 0$.
  \end{lemma}

  \section{The dimension of the bounded derived category of $(m,n)$-Igusa-Todorov algebras}

  Let $M\in \mod \Lambda$
   and an integer $i \in \mathbb{Z}$. We
  let $S^{i}(M)$ denote the stalk complex with $M$
  in the $i$th place and $0$ in the other places.
  From the proof of \cite[Theorem]{han2009derived},
  we have the following lemma.
  \begin{lemma}\label{lem2.6}{\rm (\cite{zheng2020thedimension})}
  For any bounded complex $\textbf{X}$ over $\mod \Lambda$, we have
  $$\textbf{X}\in \langle\oplus_{i\in \mathbb{Z}} S^{0}(Z_{i}(\textbf{X}))
  \rangle _{1}\diamond
  \langle\oplus_{i\in \mathbb{Z}}S^{0}(B_{i}(\textbf{X})) \rangle _{1}$$
  in $D^{b}(\mod \Lambda)$,
  and $\oplus_{i\in \mathbb{Z}} S^{0}(Z_{i}(\textbf{X}))$ and
  $\oplus_{i\in \mathbb{Z}}S^{0}(B_{i}(X)) $ have only
   finitely many nonzero summands.
  \end{lemma}

  \begin{lemma}\label{longexactsequence1}
  If we have the following exact sequence
  $$0
    \rightarrow X_{n}
    \rightarrow X_{n-1}
    \rightarrow
    \cdots
    \rightarrow X_{1}
    \rightarrow X_{0}
    \rightarrow X_{-1}
    \rightarrow 0
    $$
    in $\mod \Lambda$, then
    $\langle S^{0}(X_{-1}) \rangle_{1}
  \subseteq
  \langle \oplus_{i=0}^{n}S^{0}(X_{i}) \rangle_{n+1}.$
  \end{lemma}
  \begin{proof}
  Set $K_{i}:=\Ker (X_{i-1}\rightarrow X_{i-2})$
  for each $1 \leqslant  i \leqslant n $, and $K_{0}:=X_{-1}$.
  Now, for each $0\leqslant i \leqslant n$,
  we can get the following short exact sequences
  $$0 \rightarrow K_{i+1}  \rightarrow  X_{i}\rightarrow  K_{i}\rightarrow 0;$$
 moreover, we have the following triangles in $D^{b}(\mod \Lambda)$
  $$S^{0}(K_{i+1})  \rightarrow  S^{0}(X_{i})\rightarrow  S^{0}(K_{i})\rightarrow S^{0}(K_{i+1})[1];$$
  and then, we have
  $$ \langle S^{0}(K_{i}) \rangle_{1}
  \subseteq \langle S^{0}(X_{i}) \rangle_{1} \diamond \langle S^{0}(K_{i+1})[1] \rangle_{1}
  = \langle S^{0}(X_{i}) \rangle_{1} \diamond \langle S^{0}(K_{i+1}) \rangle_{1}.
  $$
  And by Lemma \ref{lem2.5} , we get
  $$\langle S^{0}(X_{-1}) \rangle_{1}
  \subseteq
  \langle \oplus_{i=0}^{n}S^{0}(X_{i}) \rangle_{n+1} .
  $$
  \end{proof}

  \begin{lemma}{\rm (\cite[Lemma 3.1]{aihara2015generators})}
  \label{bound_complex_X}
  Let $\textbf{X}$ be a bound complex of objects of $\mod \Lambda$.
Suppose that the homology $H_{i}(\textbf{X})$ is a projective module for every integer $i$. Then
$\textbf{X} \in \langle S^{0}(\Lambda)\rangle_{1}$.
  \end{lemma}

  \begin{theorem}\label{maintheorem1}
  Let $\Lambda$ be an $(m,n)$-Igusa-Todorov algebra.
  Then $\tridim D^{b}(\mod \Lambda) \leqslant 2m+n+1.$
  \end{theorem}
  \begin{proof}
  Take a bounded complex $\textbf{X}$ in $D^{b}(\mod \Lambda)$.
  By the construction in the proof of \cite[Proposition 3.2]{aihara2015generators}
   and \cite[Lemma 2.5]{krause2006rouquier}, we have the
   following complex
  \[\xymatrix{
  \Omega(\textbf{X}):=&\cdots \ar[r]&\Omega(X_{i+1})\oplus R_{i+1}\ar[r]\ar@{^{(}->}[d]&\Omega(X_{i})\oplus R_{i} \ar[r]\ar@{^{(}->}[d]&\Omega(X_{i-1})\oplus R_{i-1} \ar@{^{(}->}[d]\ar[r]&\cdots\\
  P^{\textbf{X}}:=&\cdots \ar[r]&P^{X_{i+1}}\ar[r]\ar@{->>}[d] &  P^{X_{i}}\ar[r]\ar@{->>}[d] &P^{X_{i-1}}\ar[r]\ar@{->>}[d] &\cdots \\
  \textbf{X}:=&\cdots \ar[r]&X_{i+1} \ar[r] &X_{i}  \ar[r]         &X_{i-1}\ar[r]&\cdots\\
   }
  \]
  where all $P^{X_{j}},R_{j}$ are projective.
  Now consider the above construction, where all $R_{i}$ are projective.
  Observe the following commutative diagrams, where
   all $Q_{i}$ are projective
  \begin{center}
  \[\xymatrix{
  \Omega(X_{i+1})\oplus R_{i+1}\ar[rrrr]\ar@{^{(}->}[dd]\ar@{->>}[rd]&&&&\Omega(X_{i})\oplus R_{i} \ar@{^{(}->}[dd]\\
  &\Omega(B_{i}(\textbf{X}))\ar@{^{(}->}[r]\ar@{^{(}->}[dd]&\Omega(Z_{i}(\textbf{X}))\oplus Q_{i}\ar@{^{(}->}[rru]\ar@{->>}[r]\ar@{^{(}->}[dd]& \Omega(H_{i}(\textbf{X}))\ar@{^{(}->}[dd]  & &  &  & &  \\
  P^{X_{i+1}}\ar[rrrr]\ar@{->>}[rd]\ar@{->>}[dd]
    &        &  &  &  P^{X_{i}}\ar@{->>}[dd]  \\
              &P^{B_{i}(\textbf{X})}\ar@{^{(}->}[r]\ar@{->>}[dd]&P^{Z_{i}(\textbf{X})}\oplus Q_{i}\ar@{^{(}->}[rru]\ar@{->>}[r]\ar@{->>}[dd] & P^{H_{i}(\textbf{X})}\ar@{->>}[dd]&    & \\
  X_{i+1} \ar[rrrr]\ar@{->>}[rd]
              &               &   &          &X_{i}\\
              &B_{i}(\textbf{X})\ar@{^{(}->}[r]                &Z_{i}(\textbf{X})\ar@{->>}[r]\ar@{^{(}->}[rru]&H_{i}(\textbf{X})         \\
   }
  \]
  \end{center}
   we can find that
  $$\Im (\Omega(X_{i+1})\oplus R_{i+1}\rightarrow \Omega(X_{i})\oplus R_{i})=\Omega(B_{i}(\textbf{X}))\in \Omega(\mod \Lambda) ,$$
  and
  $$\Ker(\Omega(X_{i+1})\oplus R_{i+1}\rightarrow \Omega(X_{i})\oplus R_{i})=\Omega(Z_{i+1}(\textbf{X}))\oplus Q_{i+1}\in \Omega(\mod \Lambda) .$$
  Inductively, we can
  get the following triangles in $D^{b}(\mod \Lambda)$
  $$\Omega(\textbf{X}) \longrightarrow P^{\textbf{X}}  \longrightarrow \textbf{X} \longrightarrow \Omega(\textbf{X})[1]$$
  $$\Omega^{2}(\textbf{X}) \longrightarrow P^{\Omega(\textbf{X})}\longrightarrow \Omega(\textbf{X}) \longrightarrow \Omega^{2}(\textbf{X})[1]$$
  $$\vdots$$
  $$\Omega^{n}(\textbf{X}) \longrightarrow P^{\Omega^{n-1}(\textbf{X})}  \longrightarrow \Omega^{n-1}(\textbf{X}) \longrightarrow \Omega^{n}(\textbf{X})[1]$$
  Since all $H_{i}(P^{\Omega^{j}(\textbf{X})})(i\in \mathbb{Z}, 0 \leqslant j \leqslant n-1)$ are projective, we know that
  $P^{\Omega^{j}(\textbf{X})}\in \langle S^{0}(\Lambda) \rangle_{1}$ by Lemma \ref{bound_complex_X}.
  Hence, we have
  \begin{align*}
   \langle \textbf{X}\rangle_{1}\subseteq  \langle  P^{\textbf{X}} \rangle_{1}\diamond \langle \Omega(\textbf{X})[1]\rangle_{1}
  \subseteq \langle  S^{0}(\Lambda) \rangle_{1}\diamond \langle \Omega(\textbf{X})\rangle_{1}.
  \end{align*}
  Repeating this process, we can get
  \begin{align*}
   \langle \Omega^{i}(\textbf{X})\rangle_{1}\subseteq  \langle  P^{\textbf{X}} \rangle_{1}\diamond \langle \Omega^{i+1}(\textbf{X})[1]\rangle_{1}
  \subseteq \langle  S^{0}(\Lambda) \rangle_{1}\diamond \langle \Omega^{i+1}(\textbf{X})\rangle_{1}.
  \end{align*}
  Moreover, and by Lemma \ref{lem2.5}, we have
  \begin{align*}
   \langle \textbf{X}\rangle_{1}
  \subseteq \langle  S^{0}(\Lambda) \rangle_{n}\diamond \langle \Omega^{n}(\textbf{X})\rangle_{1}.
  \end{align*}
  And by Lemma \ref{longexactsequence1}, we can get
  \begin{align*}
   \langle \textbf{X}\rangle_{1}
  \subseteq \langle  S^{0}(\Lambda) \rangle_{n}\diamond \langle \oplus_{i\in \mathbb{Z}} S^{0}(Z_{i}(\Omega^{n}(\textbf{X})))\rangle_{1}
  \diamond \langle \oplus_{i\in \mathbb{Z}} S^{0}(B_{i}(\Omega^{n}(\textbf{X})))\rangle_{1}.
  \end{align*}

  Set $V$ is a $(m,n)$-Igusa-Todorov module.
  Note that $Z_{i}(\Omega^{n}(\textbf{X}))=\Omega^{n}(Z_{i}(\textbf{X}))\oplus Q_{i}\in \Omega^{n}(\mod \Lambda)$,
  by Definition \ref{mnITalgebra}, we can get the following exact sequence
  $$0
    \longrightarrow V_{m}
    \longrightarrow V_{m-1}
    \longrightarrow
    \cdots
    \longrightarrow V_{1}
    \longrightarrow V_{0}
    \longrightarrow Z_{i}(\Omega^{n}(\textbf{X}))
    \longrightarrow 0
    $$
  where $V_{j}\in [V\oplus \Lambda]_{1}$ for each $0 \leqslant j \leqslant m.$
  By Lemma \ref{lem2.6}, we have
  \begin{align*}
   \langle S^{0}(Z_{i}(\Omega^{n}(\textbf{X})))\rangle_{1}
  \subseteq \langle \oplus_{j=0}^{m}S^{0}(V_{j})\rangle_{m+1}\subseteq \langle S^{0}(V)\rangle_{m+1}.
  \end{align*}
  Similarly, we have
  \begin{align*}
   \langle S^{0}(B_{i}(\Omega^{n}(\textbf{X})))\rangle_{1}\subseteq \langle S^{0}(V)\rangle_{m+1}.
  \end{align*}
  Moreover, by Lemma \ref{lem2.5}, we have
  \begin{align*}
   \langle \textbf{X}\rangle_{1}
   \subseteq&  \langle  S^{0}(\Lambda) \rangle_{n}\diamond \langle \oplus_{i\in \mathbb{Z}} S^{0}(Z_{i}(\Omega^{n}(\textbf{X})))\rangle_{1}
  \diamond \langle \oplus_{i\in \mathbb{Z}} S^{0}(B_{i}(\Omega^{n}(\textbf{X})))\rangle_{1}\\
  \subseteq& \langle  S^{0}(\Lambda) \rangle_{n}\diamond \langle S^{0}(V)\rangle_{m+1}
  \diamond \langle S^{0}(V)\rangle_{m+1}\\
  \subseteq& \langle  S^{0}(\Lambda)\oplus S^{0}(V)  \rangle_{2m+n+2}.
  \end{align*}
  And then we get $D^{b}(\mod \Lambda)=\langle  S^{0}(\Lambda)\oplus S^{0}(V)  \rangle_{2m+n+2}.$
  By Definition \ref{tri.dimenson2.1}, we have
  $\tridim D^{b}(\mod \Lambda) \leqslant 2m+n+1.$
  \end{proof}

\begin{corollary}\label{cor1}
If artin algebra $\Lambda$ is $n$-Igusa-Todorov algebra.
 Then
$\tridim D^{b}(\mod \Lambda) \leqslant n+3.$
\end{corollary}
\begin{proof}
 By Remark \ref{mnITremark}(1) and Theorem \ref{maintheorem1}.
\end{proof}

  \section{$(\ell\ell^{t_{\V}}(\Lambda)-2,\pd \V+2)$-Igusa-Todorov
  algebra}

For a module $M\in \mod\Lambda$, we use $\rad M$ and $\top M$
to denote the radical and top of $M$ respectively.
We use $\add M$ to denote the subcategory of $\mod\Lambda$
consisting of direct summands of finite direct sums of module $M$.
Let $\mathcal{V}$ be a subset of all simple modules, and $\mathcal{V}'$
the set of all the others simple modules in $\mod\Lambda$.
We write $\mathfrak{F}(\mathcal{V}):=\{M\in\mod\Lambda\;|\;\text{ there exists a chain }
0\subseteq M_{0}\subseteq M_{1}\subseteq M_{2}\subseteq \cdots\subseteq M_{m-1}\subseteq M_{m}=M
\text{ of submodules of  } M \text{ such that each quotients } M_{i}/M_{i-1}\in\mathcal{V}\}.$
$\mathfrak{F}(\mathcal{V})$ is closed under extensions, submodules and quotients modules. Then we have
a torsion part $(\mathcal{T},\mathfrak{F}(\mathcal{V}))$, and the corresponding torsion radical is
denoted by $t_{\mathcal{V}}$, and we set $q_{t_{\mathcal{V}}}(M)=M/t_{\mathcal{V}}(M)$ for each $M\in \mod\Lambda$.
Note that, $\rad,\top,t_{\mathcal{V}}$ and $q_{t_{\mathcal{V}}}$ are covariant additive functors.

\begin{definition}{\rm (\cite{huard2013layer})
The $t_{\mathcal{V}}$-radical layer length is a function
$\ell\ell^{t_{\mathcal{V}}}:\;\;\mod\Lambda \longrightarrow \mathbb{N}\cup \{\infty\}$
 via $$\ell\ell^{t_{\mathcal{V}}}(M)=\inf\{i\geqslant 0\;|\;t_{\mathcal{V}}\circ F_{t_{\mathcal{V}}}^{i}(M)=0, M\in \mod\Lambda\}$$
 where $F_{t_{\mathcal{V}}}=\rad\circ t_{\mathcal{V}}. $
 }
\end{definition}

%
%
%

  \begin{lemma}{\rm (\cite{zheng2020thedimension})}\label{lem3}
  Let $\mathcal{V}$ be a subset of the set of all pairwise non-isomorphism simple $\Lambda$-modules.
  If $M\in \mathfrak{F}(\mathcal{V})$, then $\pd M\leqslant \pd\mathcal{V}.$
  In particular, we have $\pd q_{t_{\mathcal{V}}}(M)\leqslant \pd\mathcal{V}.$
  \end{lemma}
  \begin{lemma}{\rm (\cite{zheng2020thedimension})}\label{short-exact}
  Let $\mathcal{V}$ be a subset of the set of all pairwise non-isomorphism simple $\Lambda$-modules.
  For $M\in \mod \Lambda$ and positive integer $n$.
  We have the following exact sequences
  \begin{align*}
  0 \rightarrow t_{\mathcal{V}}F^{i}_{t_{\mathcal{V}}}(M)\rightarrow &F^{i}_{t_{\mathcal{V}}}(M)\rightarrow q_{t_{\mathcal{V}}}F^{i}_{t_{\mathcal{V}}}(M) \rightarrow 0\\
  0 \rightarrow F^{i+1}_{t_{\mathcal{V}}}(M)\rightarrow &t_{\mathcal{V}}F^{i}_{t_{\mathcal{V}}}(M)\rightarrow \top t_{\mathcal{V}}F^{i}_{t_{\mathcal{V}}}(M) \rightarrow 0,
  \end{align*}
  for $0\leqslant i \leqslant n-1.$
  \end{lemma}

\begin{lemma}{\rm (\cite{huard2013layer})}\label{lem-3.17}
Let $\mathcal{V} \subseteq \{\text{simple right } \Lambda$-$modules\}$ and $M\in \mod \Lambda$. If $t_{\mathcal{V}}(M)\neq 0$
then $$\ell\ell^{t_{\mathcal{V}}}(\Omega t_{\mathcal{V}}(M))\leqslant\ell\ell^{t_{\mathcal{V} }}(\Omega t_{\mathcal{V}}(\Lambda_{\Lambda}))-1.$$
\end{lemma}

By \cite[Lemma 2.6]{zheng2020upper}, we have
\begin{lemma}\label{lem-3.18}
 Let $\mathcal{V} \subseteq \{\text{simple right } \Lambda$-$modules\}$ and $M\in \mod \Lambda$.
  If $\ell\ell^{{t_{\mathcal{V}}}}(M)\leqslant k$, then $\ell\ell^{{t_{\mathcal{V}}}}(F_{t_{\mathcal{V}}}^{k}(M))=0$.
  \end{lemma}

The following result should be well known, for convenience, we give a proof.
\begin{lemma}\label{syzygyshift}
Let $\Lambda$ be an artin algebra. Given the exact sequence
$0 \longrightarrow A \longrightarrow B\longrightarrow C\longrightarrow 0$
in $\mod\Lambda$, we can get the following exact sequence
$$0 \longrightarrow \Omega^{i+1}(C) \longrightarrow \Omega^{i}(A)\oplus Q_{i}\longrightarrow \Omega^{i}(B)\longrightarrow 0$$
for some projective module $Q_{i}$ in $\mod\Lambda$, where $i\geqslant 0.$
\end{lemma}
\begin{proof}
Consider the following pullback
  \begin{gather}\label{diagram1}
  \begin{split}
  \xymatrix{
  & & 0\ar[d]& 0 \ar[d] & \\
 & & \Omega(C) \ar @{=}[r]\ar[d] &\Omega(C)\ar[d]&\\
  0\ar[r]&A  \ar[r]\ar@{=}[d]  & A\oplus P \ar[r]\ar[d]& P \ar[r]\ar[d]&0\;\\
  0\ar[r]&A\ar[r] &B\ar[r]\ar[d]& C \ar[r]\ar[d]&0,\\
  &&0&0&  }
  \end{split}
  \end{gather}
  where $P$ is the projective cover of $C$. Applying horseshoe lemma to the
  middle column exact sequence of the above diagram, we can get the following exact sequences
  $$0 \longrightarrow \Omega^{i+1}(C) \longrightarrow \Omega^{i}(A)\oplus Q_{i}\longrightarrow \Omega^{i}(B)\longrightarrow 0$$
for some projective module $Q_{i}$ in $\mod\Lambda$, where $i\geqslant 0.$
\end{proof}
Now, we can get the following main results.

  \begin{theorem}\label{maintheorem2}
  Let $\Lambda$ be an artin algebra.
  $\V$ is the set of some simple modules with finite projective dimension.
  If $\ell\ell^{t_{\V}}(\Lambda)\geqslant 2,$
then $\Lambda$ is a $(\ell\ell^{t_{\V}}(\Lambda)-2,\pd \V+2)$-Igusa-Todorov algebra.
  \end{theorem}
  \begin{proof}
  Let $\ell\ell^{t_{\V}}(\La)=m$ and $\pd \V=n$.
  Let $M\in \mod \Lambda$ and $N=\Omega t_{\V}(M)$.
  From the following exact sequence
  \begin{align*}
  0 \rightarrow t_{\mathcal{V}}(M)\rightarrow &M\rightarrow q_{t_{\mathcal{V}}}(M) \rightarrow 0,
  \end{align*}
  and by horseshoe lemma, we can get
  \begin{align}
   \Omega^{n+2}(M) \cong \Omega^{n+2}(t_{\V}(M))=\Omega^{n+1}(N)\label{eq1}.
  \end{align}
  By Lemma \ref{short-exact}, we have the following exact sequences
  \begin{align}
  0 \rightarrow t_{\mathcal{V}}F^{i}_{t_{\mathcal{V}}}(N)\rightarrow &F^{i}_{t_{\mathcal{V}}}(N)\rightarrow q_{t_{\mathcal{V}}}F^{i}_{t_{\mathcal{V}}}(N) \rightarrow 0 \label{eq2}\\
 0 \rightarrow F^{i+1}_{t_{\mathcal{V}}}(N)\rightarrow &t_{\mathcal{V}}F^{i}_{t_{\mathcal{V}}}(N)\rightarrow \top t_{\mathcal{V}}F^{i}_{t_{\mathcal{V}}}(N) \rightarrow 0\label{eq3},
  \end{align}
  By horseshoe lemma, and short exact sequences (\ref{eq2}) and (\ref{eq3}), we can obtain the following exact sequences
  \begin{align}
   &\Omega^{n+1}(t_{\mathcal{V}}F^{i}_{t_{\mathcal{V}}}(N))\cong \Omega^{n+1}(F^{i}_{t_{\mathcal{V}}}(N))\label{eq4}\\
  0 \rightarrow \Omega^{n+1}(F^{i+1}_{t_{\mathcal{V}}}(N))\rightarrow &\Omega^{n+1}(t_{\mathcal{V}}F^{i}_{t_{\mathcal{V}}}(N))\oplus P_{i}\rightarrow \Omega^{n+1}(\top t_{\mathcal{V}}F^{i}_{t_{\mathcal{V}}}(N)) \rightarrow 0\label{eq5},
  \end{align}
  where all $P_{i}$ are projective.
  Moreover, by $(\ref{eq4})(\ref{eq5})$, for each $i\geqslant0$,
   we get the following exact sequences
  \begin{align}
  0 \rightarrow \Omega^{n+1}(F^{i+1}_{t_{\mathcal{V}}}(N))\rightarrow & \Omega^{n+1}(F^{i}_{t_{\mathcal{V}}}(N))\oplus P_{i}\rightarrow \Omega^{n+1}(\top t_{\mathcal{V}}F^{i}_{t_{\mathcal{V}}}(N)) \rightarrow 0\label{eq6}.
  \end{align}
  By $(\ref{eq6})$, let $i=0, 1$, and consider the following push out,
  \begin{gather}\label{diagram1}
  \begin{split}
  \xymatrix{
  & 0\ar[d]& 0 \ar[d] && \\
 & \Omega^{n+1}(F^{2}_{t_{\mathcal{V}}}(N)) \ar @{=}[r]\ar[d] &\Omega^{n+1}(F^{2}_{t_{\mathcal{V}}}(N))\ar[d]&&\\
  0\ar[r]&\Omega^{n+1}(F^{}_{t_{\mathcal{V}}}(N))\oplus \textcolor{red}{P_{1}} \ar[r]\ar[d]  & \Omega^{n+1}(N)\oplus P_{0}\oplus \textcolor{red}{P_{1}} \ar[r]\ar[d]& \Omega^{n+1}(\top t_{\mathcal{V}}(N))\ar[r]\ar@{=}[d]&0\;\\
  0\ar[r]&\Omega^{n+1}(\top t_{\mathcal{V}}F^{}_{t_{\mathcal{V}}}(N))\ar[r]\ar[d] &D_{0}\ar[r]\ar[d]& \Omega^{n+1}(\top t_{\mathcal{V}}(N)) \ar[r]&0.\\
&0&0& & }
  \end{split}
  \end{gather}
   by the diagram $(\ref{diagram1})$, we can get the following two exact sequences
    \begin{align}
  0 \longrightarrow \Omega^{n+1}(F^{2}_{t_{\mathcal{V}}}(N))\longrightarrow \Omega^{n+1}(N)\oplus P_{0}\oplus \textcolor{red}{P_{1}} \longrightarrow D_{0} \longrightarrow 0,\label{eq8}
  \end{align}
    \begin{align}
 0 \longrightarrow \Omega^{n+1}(\top t_{\mathcal{V}}F^{}_{t_{\mathcal{V}}}(N))\longrightarrow &D_{0}\longrightarrow \Omega^{n+1}(\top t_{\mathcal{V}}(N)) \longrightarrow 0.\label{eq9}
  \end{align}
By $(\ref{eq8})$ and $(\ref{eq6})$(for the case $i=2$), and consider the following push out,
  \begin{gather}\label{diagram2}
  \begin{split}
  \xymatrix{
  & 0\ar[d]& 0 \ar[d] && \\
 & \Omega^{n+1}(F^{3}_{t_{\mathcal{V}}}(N)) \ar @{=}[r]\ar[d] &\Omega^{n+1}(F^{3}_{t_{\mathcal{V}}}(N))\ar[d]&&\\
  0\ar[r]&\Omega^{n+1}(F^{2}_{t_{\mathcal{V}}}(N))\oplus \textcolor{red}{P_{2}} \ar[r]\ar[d]  & \Omega^{n+1}(N)\oplus P_{0}\oplus P_{1} \oplus\textcolor{red}{P_{2}} \ar[r]\ar[d]& D_{0}\ar[r]\ar@{=}[d]&0\;\\
  0\ar[r]&\Omega^{n+1}(\top t_{\mathcal{V}}F^{2}_{t_{\mathcal{V}}}(N))\ar[r]\ar[d] &D_{1}\ar[r]\ar[d]& D_{0} \ar[r]&0.\\
&0&0& & }
  \end{split}
  \end{gather}
 by the diagram $(\ref{diagram2})$, we can get the following two exact sequences
    \begin{align}
  0 \longrightarrow \Omega^{n+1}(F^{3}_{t_{\mathcal{V}}}(N))\longrightarrow \Omega^{n+1}(N)\oplus P_{0}\oplus P_{1} \oplus\textcolor{red}{P_{2}}\longrightarrow D_{1} \longrightarrow 0,\label{eq11}
  \end{align}
    \begin{align}
 0 \longrightarrow \Omega^{n+1}(\top t_{\mathcal{V}}F^{2}_{t_{\mathcal{V}}}(N))\longrightarrow &D_{1}\longrightarrow D_{0} \longrightarrow 0.\label{eq12}
  \end{align}
Continue this process, we can get the following exact sequences
    \begin{align}
  0 \longrightarrow \Omega^{n+1}(F^{j}_{t_{\mathcal{V}}}(N))\longrightarrow \Omega^{n+1}(N)\oplus (\oplus_{i=0}^{j-1} P_{i})\longrightarrow D_{j-2} \longrightarrow 0,\label{eq13}
  \end{align}
    \begin{align}
 0 \longrightarrow \Omega^{n+1}(\top t_{\mathcal{V}}F^{j-1}_{t_{\mathcal{V}}}(N))\longrightarrow &D_{j-2}\longrightarrow D_{j-3} \longrightarrow 0,\label{eq14}
  \end{align}
where $j\geqslant 2$ and $D_{-1}:=\Omega^{n+1}(\top t_{\mathcal{V}}(N)).$

From $(\ref{eq13})$$(\ref{eq13})$, we list the following exact sequences
%
    \begin{align*}
    &0 \longrightarrow \Omega^{n+1}(\top t_{\mathcal{V}}F^{}_{t_{\mathcal{V}}}(N))\longrightarrow D_{0}\longrightarrow D_{-1} \longrightarrow 0\\
  &0 \longrightarrow \Omega^{n+1}(\top t_{\mathcal{V}}F^{2}_{t_{\mathcal{V}}}(N))\longrightarrow D_{1}\longrightarrow D_{0} \longrightarrow 0\\
  &0 \longrightarrow \Omega^{n+1}(\top t_{\mathcal{V}}F^{3}_{t_{\mathcal{V}}}(N))\longrightarrow D_{2}\longrightarrow D_{1} \longrightarrow 0\\
  &\;\;\;\;\;\;\;\;\;\;\;\;\;\;\;\;\;\;\;\vdots\\
   &0 \longrightarrow \Omega^{n+1}(\top t_{\mathcal{V}}F^{m-2}_{t_{\mathcal{V}}}(N))\longrightarrow D_{m-3}\longrightarrow D_{m-4} \longrightarrow 0\\
   &0 \longrightarrow \Omega^{n+1}(F^{m-1}_{t_{\mathcal{V}}}(N))\longrightarrow \Omega^{n+1}(N)\oplus (\oplus_{i=0}^{m-2} P_{i})\longrightarrow D_{m-3} \longrightarrow 0.
    \end{align*}
  On the other hand, by Lemma \ref{lem-3.17}, we have
  $\ell\ell^{t_{\V}}(N)=\ell\ell^{t_{\V}}(\Omega t_{\V}(M))\leqslant  \ell\ell^{t_{\V}}(\Lambda)-1=m-1$.
  By Lemma \ref{lem-3.18}, we have
  $\ell\ell^{t_{\V}}({F^{m-1}_{t_{\V}}}(N))=0,$ that is,
 $F^{m-1}_{t_{\V}}(N)\in \mathfrak{F}(\mathcal{V})$.
  And by Lemma \ref{lem3},
   $\pd F^{m-1}_{t_{\V}}(N) \leqslant \pd \V=n.$
   Now, by Lemma \ref{syzygyshift} and the above exact sequences, we can get the following
   \begin{align*}
  &0 \longrightarrow\Omega^{m-2}(D_{-1}) \longrightarrow \Omega^{n+m-2}(\top t_{\mathcal{V}}F^{}_{t_{\mathcal{V}}}(N))\oplus Q_{m-2}\longrightarrow \Omega^{m-3}(D_{0}) \longrightarrow 0\\
  &0 \longrightarrow\Omega^{m-3}(D_{0}) \longrightarrow \Omega^{n+m-3}(\top t_{\mathcal{V}}F^{2}_{t_{\mathcal{V}}}(N))\oplus Q_{m-3}\longrightarrow \Omega^{m-4}(D_{1}) \longrightarrow 0\\
    &0 \longrightarrow\Omega^{m-4}(D_{1}) \longrightarrow \Omega^{n+m-4}(\top t_{\mathcal{V}}F^{3}_{t_{\mathcal{V}}}(N))\oplus Q_{m-4}\longrightarrow \Omega^{m-5}(D_{2}) \longrightarrow 0\\
  &\;\;\;\;\;\;\;\;\;\;\;\;\;\;\;\;\;\;\;\vdots\\
    &0 \longrightarrow\Omega^{}(D_{m-4}) \longrightarrow \Omega^{n+1}(\top t_{\mathcal{V}}F^{m-2}_{t_{\mathcal{V}}}(N))\oplus Q_{1}\longrightarrow D_{m-3} \longrightarrow 0\\
   &\Omega^{n+1}(N)\oplus (\oplus_{i=0}^{m-2} P_{i})\cong D_{m-3},
    \end{align*}
    where $Q_{i}$ is projective for each $1\leqslant i \leqslant m-2.$
Consider the following pullback
\[\xymatrix{
&&0\ar[d]&0\ar[d]&\\
0 \ar[r]& \Omega^{}(D_{m-4})\ar[r]\ar[d]& W\ar[r]\ar[d]& \Omega^{n+1}(N)\ar[r]\ar[d]&0\\
0 \ar[r]& \Omega^{}(D_{m-4})\ar[r]&\Omega^{n+1}(\top t_{\mathcal{V}}F^{m-2}_{t_{\mathcal{V}}}(N))\oplus Q_{1}\ar[r]\ar[d]  & \Omega^{n+1}(N)\oplus (\oplus_{i=0}^{m-2} P_{i})\ar[r]\ar[d]&0\\
        &                        & \oplus_{i=0}^{m-2} P_{i}\ar[r]\ar[d]& \oplus_{i=0}^{m-2} P_{i}\ar[r]\ar[d]&0\\
&&0&0&\\
}\]
   The short exact sequence of the middle column in the above diagram is split since $\oplus_{i=0}^{m-2} P_{i}$ is projective.
   That is, $W\in \add(\Omega^{n+1}(\top t_{\mathcal{V}}F^{m-2}_{t_{\mathcal{V}}}(N))\oplus \Lambda).$

Now, from $(\ref{eq1})$ and the above exact sequences, we can get the following long exact
\begin{align*}
  0 \longrightarrow \Omega^{m-2}(D_{-1})\longrightarrow & \Omega^{n+m-2}(\top t_{\mathcal{V}}F^{}_{t_{\mathcal{V}}}(N))\oplus Q_{m-2}\longrightarrow \Omega^{n+m-3}(\top t_{\mathcal{V}}F^{2}_{t_{\mathcal{V}}}(N))\oplus Q_{m-3} \longrightarrow
  \end{align*}
\begin{align*}
  \cdots\longrightarrow W\longrightarrow \Omega^{n+2}(M) \longrightarrow0,
  \end{align*}
 where  $$\Omega^{m-2}(D_{-1}),\Omega^{n+m-2}(\top t_{\mathcal{V}}F^{}_{t_{\mathcal{V}}}(N))\oplus Q_{m-2},\cdots,W\in \add(\Lambda \oplus (\oplus_{i=n+1}^{n+m-1}\Omega^{i}(\Lambda/\rad\Lambda))).$$
  By Definition \ref{mnITalgebra}, we know that
  $\Lambda$ is $(m-2,n+2)$-Igusa-Todorov algebra, that is,
  $(\ell\ell^{t_{\V}}(\Lambda)-2,\pd \V+2)$-Igusa-Todorov algebra.
  \end{proof}

  \begin{theorem}\label{maintheorem3}
  Let $\Lambda$ be an artin algebra.
  $\V$ is the set of some simple modules with finite projective dimension.
  If $\ell\ell^{t_{\V}}(\Lambda)\geqslant 2.$
  Then $\tridim D^{b}(\mod \Lambda) \leqslant 2\ell\ell^{t_{\V}}(\Lambda)+\pd \V-1$.
  \end{theorem}
  \begin{proof}
  By Theorem \ref{maintheorem1} and Theorem \ref{maintheorem2}, we have
  $$\tridim D^{b}(\mod \Lambda) \leqslant 2
  (\ell\ell^{t_{\V}}(\Lambda)-2)+(\pd \V+2)+1
  =2\ell\ell^{t_{\V}}(\Lambda)+\pd \V-1.$$
  \end{proof}
  \section{Examples}
\begin{example}
{\rm (\cite{zheng2020upper})
Consider the bound quiver algebra $\Lambda=kQ/I$, where $k$ is an algebraically closed field and $Q$
is given by
$$
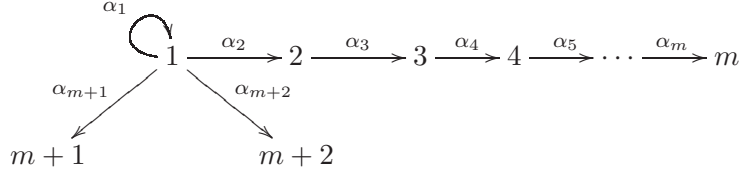
\xymatrix{
&1 \ar@(l,u)^{\alpha_{1}}\ar[r]^{\alpha_{2}}  \ar[ld]_{\alpha_{m+1}}\ar[rd]^{\alpha_{m+2}}
&2\ar[r]^{\alpha_{3}}&{3}\ar[r]^{\alpha_{4}}  &{4}\ar[r]^{\alpha_{5}}&\cdots\ar[r]^{\alpha_{m}}&m\\
m+1&&m+2&&&&
}$$
and $I$ is generated by
$\{\alpha_{1}^{2},\alpha_{1}\alpha_{m+1},\alpha_{1}\alpha_{m+2},\alpha_{1}\alpha_{2},
\alpha_{2}\alpha_{3}\cdots\alpha_{m}\}$ with $m\geqslant 10$.
Then the indecomposable projective $\Lambda$-modules are
$$
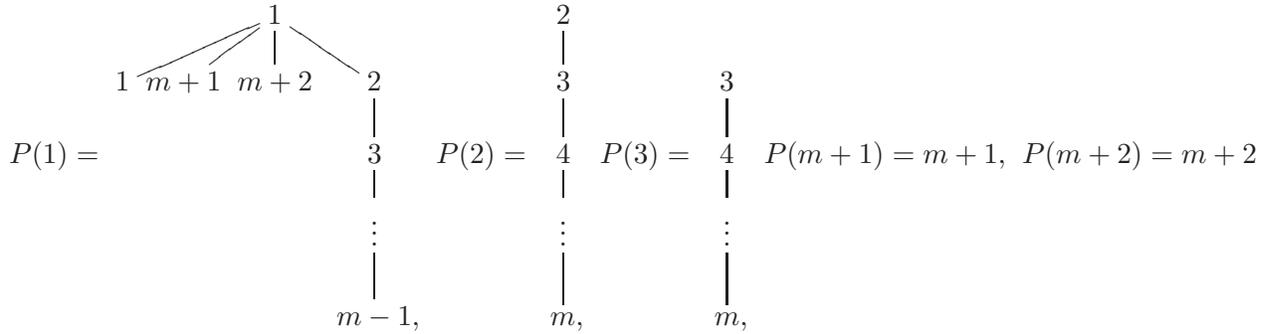
\xymatrix@-1.0pc@C=0.1pt
{ &  &  &1\edge[lld]\edge[ld]\edge[d]\edge[dr]
&  && 2\edge[d] &&&& & &&&&  &&&&  &\\
&1& m+1&m+2&2\edge[d] && 3\edge[d] && 3\edge[d] &&&& &&&& &\\
P(1)=&  &  &    &3\edge[d] &P(2)=&4\edge[d]  &P(3)=&4\edge[d] &P(m+1)=m+1,&P(m+2)=m+2&\\
&  &  &  &\vdots\edge[d]&&\vdots\edge[d]&&\vdots\edge[d]&& &&&& &\\
&  &  &  & \;m-1, &&\;m, && \;m,  &&&& &&&\\
&  &  &  &  &&&&  & &&&& & &&&& &&&\\
}$$
and $P(i+1)=\rad P(i)$ for any $2 \leqslant i\leqslant m-1$.

We have
\begin{equation*}
\pd S(i)=
\begin{cases}
\infty, &\text{if}\;\;i=1;\\
1,&\text{if} \;\;2 \leqslant  i\leqslant m-1;\\
0,&\text{if}\;\; m \leqslant  i\leqslant m+2.
\end{cases}
\end{equation*}
So $\mathcal{S}^{\infty}=\{S(1)\}$ and $\mathcal{S}^{<\infty}=\{ S(i)\mid 2\leqslant i\leqslant m+2\}$.

Let $\mathcal{V}:=\{S(i)\mid 3\leqslant i \leqslant m-1\}\subseteq\mathcal{S}^{<\infty}$.
 Then $\pd\V =1$ and $\ell\ell^{t_{\mathcal{V}}}(\Lambda)=2$(see \cite[Example 4.1]{zheng2020upper})

(1) By Theorem \ref{thm1.1}(1), we have $\tridim D^{b}(\mod \Lambda) \leqslant \LL(\Lambda)-1=m-2.$

(2) By Theorem \ref{thm1.1}(3), we have
  $\tridim D^{b}(\mod \Lambda) \leqslant (\pd\mathcal{V}+2)(\ell\ell^{t_{\mathcal{V}}}(\Lambda)+1)-2=7.$

(3) By Theorem \ref{thm1.1}(4), we have
  $\tridim D^{b}(\mod \Lambda) \leqslant 2(\pd\mathcal{V}+\ell\ell^{t_{\mathcal{V}}}(\Lambda))+1=7.$

(4) By Theorem \ref{thm1.1}(5),
  $\tridim D^{b}(\mod \Lambda) \leqslant \pd\mathcal{V}+3=4.$

(5) By Theorem \ref{maintheorem3},
  $\dim D^{b}(\mod \Lambda) \leqslant 2\ell\ell^{t_{\V}}(\Lambda_{\Lambda})+\pd\mathcal{V}-1=4.$
In fact, if $\ell\ell^{t_{\V}}(\Lambda_{\Lambda})=2$, then the upper
bound $2\ell\ell^{t_{\V}}(\Lambda_{\Lambda})+\pd\mathcal{V}-1$ is equal to
$\pd\mathcal{V}+3$.
}
\end{example}

  \begin{example}{\rm (\cite{zheng2020thedimension})}
  {\rm Let $k$ be an algebraically closed field and $\Lambda=kQ/I$, where $Q$ the quiver
  
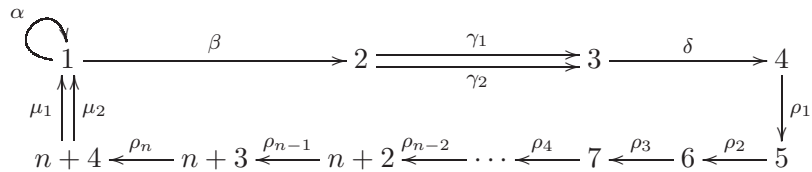
\[\xymatrix{
  &{1}\ar@(l,u)^{\alpha} \ar[rr]^{\beta} &&2 \ar@<0.5ex>[rr]^{\gamma_{1}}\ar@<-0.5ex>[rr]_{\gamma_{2}}&& 3 \ar[rr]^{\delta}&&4 \ar[d]^{\rho_{1}}&\\
  &n+4\ar@<0.5ex>[u]^{\mu_{1}} \ar@<-0.5ex>[u]_{\mu_{2}} &n+3\ar[l]_{\rho_{n}}&n+2\ar[l]_{\rho_{n-1}}&\cdots\ar[l]_{\rho_{n-2}}&7\ar[l]_{\rho_{4}}&6\ar[l]_{\rho_{3}}&5\ar[l]_{\rho_{2}}
  }\]
  and $I$ is generated by $\{\alpha^{m},\alpha\beta,
  \gamma_{1}\delta-\gamma_{2}\delta, \rho_{n}\mu_{1}\alpha, \rho_{n}\mu_{2}\alpha, \mu_{1}\beta-\mu_{2}\beta\}$, $n>2m+1$,
  $m\geqslant 5$.

  Then the indecomposable projective $\Lambda$-modules are
  $$
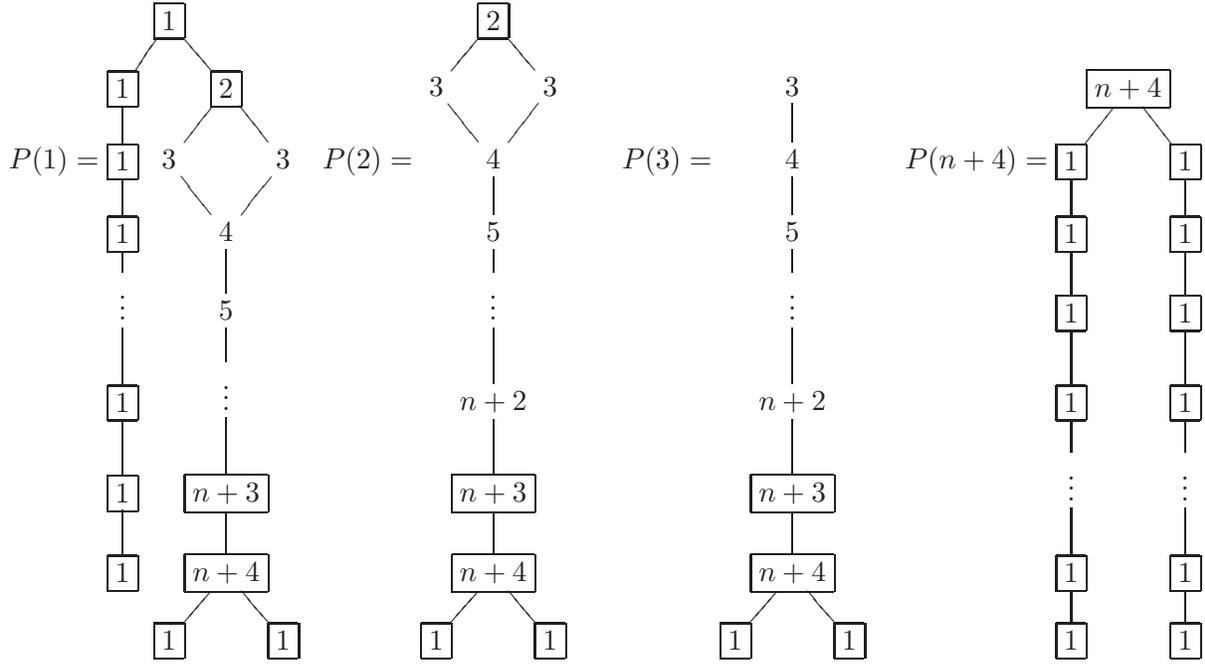
\xymatrix@-1.0pc@C=0.1pt
  {    &  &  &*+[F]{1}\edge[lld]\edge[dr]&& &&&&*+[F]{2}\edge[ld]\edge[rd]&& & &&&&  &&&&  & &  \\
  &*+[F]{1}\edge[d] &&&*+[F]{2}\edge[ld]\edge[rd]&& &&3\edge[rd]&& 3\edge[ld]&&& &&&3\edge[d]&&&&&&*+[F]{n+4}\edge[ld]\edge[rd]&&&   \\
  P(1)=&*+[F]{1}\edge[d] &&3\edge[rd] &&3\edge[ld]& &P(2)=&&4\edge[d] && &&&P(3)=& &4\edge[d]  && & & P(n+4)=&*+[F]{1}\edge[d]  &  &*+[F]{1} \edge[d] ,  \\
  &*+[F]{1}\edge[d] && &4\edge[d]&&&&&5\edge[d]&&&&&& &5\edge[d]&&& & &*+[F]{1}\edge[d]  &  &*+[F]{1} \edge[d] & &&& &  & \\
  &\vdots\edge[d]  &&&5\edge[d]&& &&&\vdots\edge[d]&&&& &&&\vdots\edge[d]&&&&&*+[F]{1}\edge[d]  &  &*+[F]{1} \edge[d]& &&&&&&& &  & \\
  &*+[F]{1}\edge[d] &&&\vdots\edge[d] &&&&&n+2\edge[d]  & &&&& & &n+2\edge[d]&&& &&*+[F]{1}\edge[d]  &  &*+[F]{1} \edge[d]&&&&& &  &&&& &  &  \\
       &*+[F]{1}\edge[d] &&&*+[F]{n+3}\edge[d]   &&&&& *+[F]{n+3}\edge[d] & && && & &*+[F]{n+3}\edge[d]&&& &&\vdots\edge[d]  &  &\vdots \edge[d]&&&&& &  & &&& &  & \\
       &*+[F]{1} &&&*+[F]{n+4}\edge[ld]\edge[rd]  &&&&&*+[F]{n+4}\edge[ld]\edge[rd]  & &&&& & & *+[F]{n+4}\edge[ld]\edge[rd]&&& &&*+[F]{1}\edge[d]  &  &*+[F]{1} \edge[d]&&&&& &  & &&& &  & \\
       &&  &*+[F]{1}&&*+[F]{1}&&&*+[F]{1}&  &*+[F]{1}&&&& &*+[F]{1}&&*+[F]{1}&& &&*+[F]{1}  &  &*+[F]{1} &&&&& &  & \\
  }$$
  and $P(i+1)=\rad P(i)$ for each $3 \leqslant i\leqslant n+2$.

  It is straightforward to verify that
  \begin{equation*}
  \pd S(i)=
  \begin{cases}
  \infty, &\text{if}\;\; i=1, n+3;\\
  2,&\text{if} \;\;i=2,n+4;\\
  1,&\text{if} \;\; 3 \leqslant  i\leqslant n+2.
  \end{cases}
  \end{equation*}
  Let $\mathcal{V}:=\{S(i)\mid 3 \leqslant  i\leqslant n+2 \text{ or } i=n+2\}$. Then
  $\pd\mathcal{V}=2.$
  Let $\mathcal{V}'$ be all the others simple modules in $\mod \Lambda$, that is,
  $\mathcal{V}'=\{ S(1),S(2), S(n+3),S(n+4)\}$.

  Because $\La=\oplus_{i=1}^{n+4}P(i)$, we have
  $$\ell\ell^{t_{\V}}(\Lambda_{\Lambda})=\max\{\ell\ell^{t_{\V}}(P(i)) \mid 1 \leqslant i  \leqslant n+4\}$$
  by \cite[Lemma 3.4(a)]{huard2013layer}.

  In order to compute $\ell\ell^{t_{\V}}(P(1))$, we need to find the least non-negative integer $i$
  such that $t_{\V}F_{t_{\V}}^{i}(P(1))=0$.
  Since $\top P(1)=S(1)\in \add \V'$, we have $t_{\V}(P(1))=P(1)$ by \cite[Proposition 5.9(a)]{huard2013layer}.
  Thus
  \begin{align*}\xymatrix@-1.0pc@C=0.1pt{
  &F_{t_{\V}}(P(1))=\rad t_{\V}(P(1))=\rad (P(1))=T_{m-1}\oplus P(2)\\
  }\end{align*}
  \xymatrix@-1.0pc@C=0.05pt {
  & *+[F]{1}\edge[d]             \\
  {\rm where}  \;\;\;T_{m-1}=  &*+[F]{1}\edge[d]&({\rm the \;number\; of\; 1 \;is \;}m-1).\;\\
  &\vdots\edge[d] \\
  & *+[F]{1}.}

  Since $\top T_{m-1}=S(1)\in \V'$, we have $t_{\V}(T_{m-1})=T_{m-1}$ by \cite[Proposition 5.9(a)]{huard2013layer}.
  Similarly, $t_{\V}(P(2))=P(2)$.
  We have
  \begin{align*}\xymatrix@-1.0pc@C=0.1pt{
  &t_{\V}F_{t_{\V}}(P(1))= t_{\V}(T_{m-1}\oplus P(2))= t_{\V}(T_{m-1})\oplus t_{\V}( P(2))=T_{m-1}\oplus P(2)
  }\end{align*}
  and
  \begin{align*}\xymatrix@-1.0pc@C=0.1pt{
  &F^{2}_{t_{\V}}(P(1))&=\rad t_{\V}F_{t_{\V}}(P(1))=\rad(T_{m-1}\oplus P(2))=\rad(T_{m-1})\oplus \rad(P(2))=T_{m-2}\oplus M
  .
  }\end{align*}

  \xymatrix@-1.0pc@C=0.05pt {
  &3\edge[rd]&&3\edge[ld]\\
  && 4\edge[d]  \\
  {\rm where}  \;\;\;M= && 5\edge[d]  \\
  && \vdots\edge[d]  \\
  && n+2\edge[d]  \\
  && *+[F]{n+3}\edge[d]  \\
  && *+[F]{n+4}\edge[ld]\edge[rd]  \\
  &*+[F]{1}&&*+[F]{1}
  .}
  And \begin{align*}\xymatrix@-1.0pc@C=0.1pt{
  &t_{\V}F^{2}_{t_{\V}}(P(1))&=t_{\V}(T_{m-2}\oplus M)=t_{\V}(T_{m-2})\oplus t_{\V}(M)=T_{m-2}\oplus P(n+3)
  .
  }\end{align*}

  Repeat the process,
  we have
  can get that $S(1)$ is a direct summand of $t_{\V}F^{m-1}_{t_{\V}}(P(1))$, that is $t_{\V}F^{m-1}_{t_{\V}}(P(1))\neq 0$;
  and
  $t_{\V}F^{m}_{t_{\V}}(P(1))=0.$ Moreover, $\ell\ell^{t_{\V}}(P(1))=m$.
  Similarly, we have
  \begin{equation*}
  \ell\ell^{t_{\V}}(P(i))=
  \begin{cases}
  4,&\text{if} \;\;2\leqslant  i\leqslant n;\\
  3,&\text{if} \;\;3\leqslant  i\leqslant  n+3;\\
  m+1,&\text{if} \;\;i= n+4.
  \end{cases}
  \end{equation*}
  Consequently, we conclude that $\ell\ell^{t_{\V}}(\Lambda_{\Lambda})=\max\{\ell\ell^{t_{\V}}(P(i))\mid 1 \leqslant i  \leqslant  n+4\}=m+1$.
  We have $\ell\ell^{t_{\mathcal{V}}}(\Lambda_{\Lambda})=m.$

  Note that  $\LL(\Lambda)=n+5$ and $\gldim \Lambda=\infty$.

(1) By Theorem \ref{thm1.1}(1), we have $$\tridim D^{b}(\mod \Lambda) \leqslant \LL(\Lambda)-1=n+4.$$

(2) By Theorem \ref{thm1.1}(3), we have
  $$\tridim D^{b}(\mod \Lambda) \leqslant (\pd\mathcal{V}+2)(\ell\ell^{t_{\mathcal{V}}}(\Lambda)+1)-2=(1+2)(m+1+1)-2=3m+4.$$

(3) By Theorem \ref{thm1.1}(4), we have
  $$\tridim D^{b}(\mod \Lambda) \leqslant 2(\pd\mathcal{V}+\ell\ell^{t_{\mathcal{V}}}(\Lambda))+1=2\times(2+m)+1=2m+5.$$

(4) By Theorem \ref{maintheorem3},
  $\tridim D^{b}(\mod \Lambda) \leqslant 2\ell\ell^{t_{\V}}(\Lambda_{\Lambda})+\pd\mathcal{V}-1=2m+1-1=2m.$

  Since $n>2m+1$ and $m\geqslant 5$, we get that
  $$2m<2m+5 =\inf\{ 2(\pd\mathcal{V} +\ell\ell^{t_{\mathcal{V}}}(\Lambda))+1,\gldim \Lambda,
   \LL(\Lambda)-1,(\pd\mathcal{V} +2)(\ell\ell^{t_{\mathcal{V}}}(\Lambda)+1)-2 \} .$$
  That is, our new upper bound for the dimension of the dimenison of $D^{b}(\mod \Lambda)$ can get a better upper bound of bounded derived categories sometimes.

  }
  \end{example}

\vspace{0.6cm}

{\bf Acknowledgements.}
The author would like to thank Dong Yang for his reading of this paper
and for his comments.
This work was supported by the National Natural Science Foundation of China(Grant No. 12001508).


%
\bibliographystyle{abbrv}

\end{document}